\documentclass[10pt]{amsart}
\usepackage[foot]{amsaddr}
\usepackage{amsmath}
\usepackage{amsfonts}
\usepackage{amssymb}
\usepackage{amsthm}

\usepackage{url}
\usepackage{tikz-cd}
\usepackage{dsfont}
\usepackage{graphicx}
\usepackage{caption}
\usepackage{subcaption}

\usepackage{comment}
\usepackage[colorlinks,pagebackref,hypertexnames=false,
pdftitle={Toric sheaves and flips}]{hyperref}
\usepackage{color}
\usepackage{enumerate}
\usepackage{mathrsfs}
\usepackage{enumitem}

\usepackage{todonotes}

\numberwithin{equation}{section}

\allowdisplaybreaks

\newtheorem{proposition}{Proposition}[section]
\newtheorem{lemma}[proposition]{Lemma}
\newtheorem{theorem}[proposition]{Theorem}

\theoremstyle{definition}
\newtheorem{remark}[proposition]{Remark}
\newtheorem{definition}[proposition]{Definition}

\newcommand{\rk}{\mathrm{rank}}

\newcommand{\Pic}{\mathrm{Pic}}

\def\K{\mathbb{K}}
\def\R{\mathbb{R}}
\def\Q{\mathbb{Q}}
\def\Z{\mathbb{Z}}
\def\N{\mathbb{N}}

\def\P{\mathbb{P}}

\def\C{\mathbb{C}}

\def\cF{\mathcal{F}}
\def\cE{\mathcal{E}}
\def\cO{\mathcal{O}}

\def\cG{\mathcal{G}}

\def\cT{\mathcal{T}}

\def\cR{\mathcal{R}}

\def\Obj{\mathrm{Obj}}
\def\Gr{\mathrm{Gr}}

\def\rank{\mathrm{rank}}

\def\fE{\mathfrak{E}}

\def\Kl{\mathfrak{Kl}}

\def\Filt{\mathfrak{Filt}}

\def\ep{\varepsilon}
\def\>{\rangle}

\def\<{\langle}
\def\>{\rangle}

\def\Cone{\mathrm{Cone}}
\def\mult{\mathrm{mult}}
\def\Span{\mathrm{Span}}

\title[Toric sheaves and flips]{Toric sheaves and flips}

\author[A. Clarke]{Andrew Clarke}
\address{Andrew Clarke, Instituto de Matem\'atica, Universidade Federal do Rio de Janeiro, Av. Athos da Silveira Ramos 149, Rio de Janeiro, RJ, 21941-909, Brazil}
\email{andrew@im.ufrj.br} 

\author[A. Napame]{Achim NAPAME}

\address{Achim Napame, Instituto de Matem{\'a}tica, Estat{\'i}stica e Computa{\c c}{\~a}o Cient{\'i}fica - UNICAMP, Rua S{\'e}rgio Buarque de Holanda 651, 13083-970 Campinas-SP, Brazil}
\email{achim@unicamp.br}

\author[C. Tipler]{Carl Tipler}
\address{Carl Tipler, Univ Brest, UMR CNRS 6205, Laboratoire de Math\'ematiques de Bretagne Atlantique, France}
\email{Carl.Tipler@univ-brest.fr}


\begin{document}
\def\Ref{\mathfrak{Ref}}

\begin{abstract}
 Any toric flip naturally induces an equivalence between the associated categories of equivariant reflexive sheaves, and we investigate how slope stability behaves through this functor. On one hand, for a fixed toric sheaf, and natural polarisations that make the exceptional loci small, we provide a simple numerical criterion that characterizes when slope stability is preserved through the flip. On the other hand, for a given flip, we introduce full subcategories of logarithmic toric sheaves and characterize when polystability is preserved for all toric sheaves in those subcategories at once.
\end{abstract}

\maketitle

\section{Introduction}
\label{sec:intro}

Introduced by Mumford \cite{Mum62} and generalized by Takemoto \cite{Tak72}, slope stability of vector bundles, and more generally of torsion-free coherent sheaves, can be used as a device to produce moduli spaces. While slope stability is not a GIT notion in higher dimension, it behaves well with respect to tensor products and restrictions, and has a differential geometric interpretation in gauge theory through the Hitchin--Kobayashi correspondence (see e.g. \cite{Kob87} and references therein). In particular, stable bundles, and more generally stable reflexive sheaves (\cite{BS}), are of particular interest in gauge theory and mathematical physics (see e.g. \cite{KnuSha98} for a survey on stable sheaves on toric varieties addressed to the mathematical physics community). 

Despite its usefulness, checking stability in practice remains a difficult problem. Our goal is to add to the list of known methods to produce stable sheaves via transformations of the underlying polarised manifold. In the equivariant context of toric geometry, the behaviour of slope stability through descent under GIT quotients was studied in \cite{CT22}, while the problem of pulling back stable sheaves on fibrations was considered in \cite{NapTip} (note though that stability is not necessary to produce moduli spaces of equivariant toric bundles, cf \cite{payne2008}). In this paper, we study how slope stability is affected through a toric flip between polarised (simplicial) toric  varieties. Those transformations are of particular interest for several reasons. From the complex geometry point of view, they form building blocks for the toric Minimal Model Program (see \cite[Ch. 15]{CLS} and references therein), and together with fibrations and blow-ups addressed in \cite{NapTip}, our results complete the study of slope stability through any type of extremal contraction arising in toric MMP. From the mathematical physics perspective, toric flips can be seen as singular transitions between toric varieties.
Those transitions are of particular importance given the construction of Calabi--Yau hypersurfaces in toric Gorenstein Fano varieties (\cite{Batyrev94}) and the connections between various Calabi--Yau vacua through conifold transitions (\cite{Candelas,Reid}). Our results then provide a toy model in the study of stable sheaves through singular transitions between toric varieties (see \cite{CPY} for a differential geometric approach to stability of the tangent bundle through conifold transitions).

Consider a toric flip (see Section \ref{sec:toricflips} for the precise definitions):
$$
 \begin{tikzcd}
  X   &  & X' \arrow[ddl, rightarrow,"{\phi'}"']\arrow[ll,dashleftarrow, "{\psi}"'] \\
   & & \\
  & X_0 \arrow[uul, leftarrow, "{\phi}"'] &  
\end{tikzcd}
 $$
between two simplicial toric varieties $X$ and $X'$.
There is a $\Q$-Cartier divisor $D_+\subset X$ naturally attached to the flip, such that $-D_+$ is $\phi$-ample and restricts to the anticanonical divisor of the $\phi$-contracted fibers (see Section \ref{sec:exceptionalloci}). By abuse of notation, we will still denote $D_+$ the divisor $\psi_*(D_+)\subset X'$. Then, for some ample Cartier divisor $L_0$ on $X_0$, there exists $\ep_0 >0$ such that the divisors
$$
L_{-\ep}:=\phi^*L_0 - \ep D_+\subset X
$$
and 
$$
L_{\ep}:=(\phi')^*L_0 + \ep D_+\subset X'
$$
define ample $\Q$-Cartier divisors for $\ep\in (0,\ep_0)$. Then, our first result (Theorem \ref{theo:main}) relates slope stability of toric sheaves on $(X_0, L_0)$, $(X, L_{-\ep})$ and $(X', L_\ep)$ :

\begin{theorem}
 \label{theo:mainintro}
 Let $\cE$ be a torus equivariant reflexive sheaf on $X$. Then, up to shrinking $\ep_0$, we have for all $\ep\in(0,\ep_0)$ :
 \begin{enumerate}
  \item[(i)] If $\phi_*\cE$ is $L_0$-stable, then $\cE$ (resp. $\psi_*\cE$) is $L_{-\ep}$-stable on $X$ (resp. $L_\ep$-stable on $X'$).
  \item[(ii)] If  $\phi_*\cE$ is $L_0$-unstable, then $\cE$ (resp. $\psi_*\cE$) is $L_{-\ep}$-unstable on $X$ (resp. $L_\ep$-unstable on $X'$).
  \item[(iii)] If $\phi_*\cE$ is $L_0$-semistable, let $\fE$ be the finite set of equivariant and saturated reflexive subsheaves $\cF \subseteq \phi_* \cE$ appearing in a Jordan--H\"older filtration of $\phi_* \cE$. If for every $\cF \in \fE$,
  \begin{equation}
   \label{eq:numericalcriterion}
  \dfrac{c_1(\cE) \cdot D_{+} \cdot (\phi^* L_0)^{n-2}}{\rk(\cE)} <
  \dfrac{c_1((\phi^*\cF)^{\vee \vee}) \cdot D_{+} \cdot (\phi^* L_0)^{n-2}}{\rk(\cF)}
  \end{equation}
  then $\cE$  is $L_{-\ep}$-stable on $X$.
 \end{enumerate}
\end{theorem}

The statements $(i)$ and $(ii)$ also follow from the openness of stability \cite[Theorem 3.3]{GKPeternell15}. Indeed, if $\phi_{*}\cE$ is $L_0$-stable (resp. $L_0$-unstable), then by the projection formula, $\cE$ is stable (resp. unstable) with respect to the nef and big divisor
$\phi^* L_0$.

A similar statement as $(iii)$ holds for $\psi_*\cE$ with the reverse inequalities in (\ref{eq:numericalcriterion}) (see Theorem \ref{theo:main}). The intersection number that appears in  (\ref{eq:numericalcriterion}) is the first order term in the $\ep$-expansion of the $L_{-\ep}$-slope. As Theorem \ref{theo:mainintro} shows, if $\phi_*\cE$ is strictly semistable, this term will never allow for both $\cE$ and $\psi_*\cE$ to be stable at the same time, for the considered polarisations.  As $\fE$ is finite, the numerical criterion in $(iii)$ can be used in practice to produce examples of toric sheaves that go from being unstable to stable through a toric flip (see Section \ref{sec:example}). 
\begin{remark}
 Theorem \ref{theo:mainintro} actually holds uniformly for flat families of toric sheaves with fixed characteristic function (see the discussion at the end of Section \ref{sec:main}).
\end{remark}
While our first result focuses on specific flat families of sheaves for a given flip, our second result describes toric flips that preserve slope polystability for {\it all} equivariant reflexive sheaves at once, in some adapted full subcategories. Denote by $\Ref^T(X)$ the category of torus equivariant reflexive sheaves on $X$. For any given torus invariant divisor $D\subset X$, we introduce in  Section \ref{sec:logcase} a full subcategory $\Ref^T(X,D)$  of {\it logarithmic toric sheaves}. Our terminology is inspired by the fact that the logarithmic tangent sheaf $\cT_X(-\log D)$ belongs to $\mathrm{Obj}(\Ref^T(X,D))$. Setting $D'=\psi_*D$, the birational transformation $\psi : X \dashrightarrow X'$ induces an equivalence of categories (still denoted $\psi_*$) between $\Ref^T(X,D)$ and $\Ref^T(X',D')$. We will say that the functor $\psi_*$ preserves polystability for a pair of ample classes $(\alpha,\alpha')\in \Pic(X)_\Q\times \Pic(X')_\Q$ if for any $\cE\in \mathrm{Obj}(\Ref^T(X,D))$, $\cE$ is polystable on $(X,\alpha)$ if and only if $\psi_*\cE$ is polystable on $(X',\alpha')$. Then, our result is as follows (see Theorem \ref{theo:second}) :

\begin{theorem}
 \label{theo:secondintro}
 Let $\Sigma$ be the fan of $X$, and let $D$ be the torus invariant divisor
 $$
 D:=\sum_{\rho\in \Delta} D_\rho \subset X,
 $$
  for $\Delta\subset \Sigma(1)$. Then, the following assertions are equivalent, for a pair of ample classes $(\alpha,\alpha')\in \Pic(X)_\Q\times \Pic(X')_\Q$.
 \begin{enumerate}
  \item[(i)] The functor $\psi_* : \Ref^T(X,D) \to \Ref^T(X',D)$ preserves polystability for $(\alpha,\alpha')$.
  \item[(iii)] There is $c \in\Q_{>0}$ such that for all $\rho\notin \Delta$, $$\deg_\alpha D_\rho = c\; \deg_{\alpha'} D_\rho'.$$
 \end{enumerate}
\end{theorem}
In the above statement, $D_\rho$ stands for the torus invariant divisor associated to a ray $\rho \in\Sigma(1)$, and 
$ \deg_\alpha$ stands for the degree of a divisor on $(X,\alpha)$. We should point out that condition $(iii)$ becomes very restrictive when $\Delta$ is small, while $\Ref^T(X,D)$ becomes smaller for larger $\Delta$'s. Nevertheless, Theorem \ref{theo:secondintro} provides a simple numerical criterion for pairs of classes on $X$ and $X'$ to preserve polystability of  specific toric sheaves through the toric flip.

\begin{remark}
 Our approach to prove Theorem \ref{theo:mainintro} and Theorem \ref{theo:secondintro} uses Klyachko's description of toric sheaves and Kool's formula for the slope of such objects (cf \cite{Kly90} and \cite{Koo11}). It would be interesting to see how the recent work by Devey \cite{Dev22} on stable toric sheaves can be used to approach those results.
\end{remark}

\subsection*{Acknowledgments}  
The authors would like to thank the anonymous referees for their helpful advice.
The authors are partially supported by the grant BRIDGES ANR--FAPESP ANR-21-CE40-0017. AN and CT also benefited from the grant MARGE ANR-21-CE40-0011. AC is also partially supported by the grant { Projeto CAPES - PrInt UFRJ 88887.311615/2018-00}.

\subsection*{Notations and conventions}
All varieties we will consider will be normal toric varieties over the complex numbers. For such a variety $X$, we denote by $T$ its torus, $N$ its lattice of one parameter subgroups, $M$ its lattice of characters. For $\K$ a field of characteristic zero, and a lattice $W$, we denote $W_\K:=W\otimes_\Z \K$. The fan of $X$ will be denoted $\Sigma_X$, or simply $\Sigma$ when the situation is clear enough, and we will also use the notation $X_\Sigma$ for the variety associated to the fan $\Sigma$. For a given cone $\sigma\in\Sigma$, we let $U_\sigma=\mathrm{Spec}^m(\C[\sigma^\vee\cap M])$ the affine chart in $X$, $O(\sigma)\subset X$ the orbit associated to $\sigma$ by the orbit-cone correspondence, and $V(\sigma)$ the closure in $X$ of $O(\sigma)$. Finally, for a fan $\Sigma$ and a subset $S\subset N_\R$, we denote by $\Sigma_S=\lbrace \sigma\in\Sigma\:\vert\: \sigma\subset S \rbrace$.

\section{Background on toric flips}
\label{sec:toricflips}
\subsection{Toric flips}
\label{sec:flips}
We recall in this section the basics on toric flips (and refer the reader to \cite[Section 3.3]{CLS} for the definition of toric morphisms and their fan description). While our presentation differs slightly from \cite[Chapter 15]{CLS}, we will keep most of the notations from this book, and the properties that we list here can be recovered from \cite[Lemma 15.3.7, Lemma 15.3.11 and Theorem 15.3.13]{CLS}. Let $N$ be a rank $n$ lattice, with $n\geq 3$.
\begin{definition}
\label{def:flippingcone}
 A full dimensional strictly convex cone $\sigma_0\subset N_\R$ will be called a {\it flipping cone} if there exist primitive elements $\lbrace \nu_1,\ldots,\nu_{n+1} \rbrace \subset N$ such that 
\begin{enumerate}
\item The cone $\sigma_0$ is spanned by the $\nu_i$'s : 
$$\sigma_0=\mathrm{Cone}(\nu_1,\ldots,\nu_{n+1}),$$
\item there exists 
$(b_1,\ldots,b_{n+1})\in\Z^{n+1}$ such that $$\sum_{i=0}^{n+1} b_i \nu_i =0,
$$
\item the sets $J_-=\lbrace i\:\vert\: b_i <0 \rbrace$ and $J_+=\lbrace i\:\vert\: b_i >0 \rbrace$ both contain at least $2$ elements.
\end{enumerate}
\end{definition}
For a given flipping cone $\sigma_0\subset N_\R$ as in Definition \ref{def:flippingcone}, we will denote $$J_0=\lbrace i\:\vert\: b_i =0 \rbrace.$$
We also introduce the notation, for any $J\subset \lbrace 1,\ldots, n+1\rbrace$,
$$
\sigma_J=\mathrm{Cone}(\nu_i\:\vert\: i\in J)
$$
 together with the fans 
$$
\Sigma_-=\lbrace \sigma_J\:\vert\: J_+\nsubseteq J \rbrace
$$
and 
$$
\Sigma_+=\lbrace \sigma_J\:\vert\: J_-\nsubseteq J \rbrace.
$$
Identifying $\sigma_0$ with the fan of its faces, we see that $\Sigma_-$ and $\Sigma_+$ provide refinements of $\sigma_0$. Those refinements induce toric morphisms $\phi_\pm : X_{\Sigma_\pm} \to U_{\sigma_0}$, whose properties are listed below (see \cite[Lemma 15.3.11.(c)]{CLS}) :
\begin{lemma}
 \label{lem:localtoricflip}
 The morphisms $\phi_\pm : X_{\Sigma_\pm} \to U_{\sigma_0}$ are surjective and birational. Their exceptional loci are given by $V(\sigma_{J_\pm})\subset X_{\Sigma_\pm}$ and satisfy $\phi_\pm(V(\sigma_{J_\pm}))=V(\sigma_{J_-\cup J_+})$ and $\dim V(\sigma_{J_\pm})=n-\vert J_\pm \vert$, while $\dim V(\sigma_{J_-\cup J_+})=\vert J_0 \vert$.
\end{lemma}

We now introduce the notion of toric flips that we will use in this paper.
If $W$ is a subset of the support of a fan $\Sigma$, we define the restriction $\Sigma_{\vert W}$ by
$$
\Sigma_{\vert W} = \lbrace \sigma \in \Sigma : \sigma \subset W \rbrace.
$$

\begin{definition}
 \label{def:toricflip}
Let $X$ and $X'$ be two $n$-dimensional simplicial toric  varieties with fans $\Sigma$ and $\Sigma'$ and common lattice of one-parameter subgroups $N$. We will say that they are related by a {\it toric flip} if there exists a normal toric variety $X_0$ with fan $\Sigma_0$ containing  a flipping cone $\sigma_0\in\Sigma_0$ such that
 $$
 \Sigma_{\vert \sigma_0}=\Sigma_+,\: \Sigma_{\vert \sigma_0}'=\Sigma_-\: \textrm{ and } 
 \Sigma_{\vert N_\R\setminus \sigma_0}= \Sigma_{\vert N_\R\setminus \sigma_0}'=( \Sigma_0)_{\vert N_\R\setminus \sigma_0}.
 $$
 In this situation, the refinements $\Sigma$ and $\Sigma'$ of $\Sigma_0$ induce toric morphisms $\phi : X\to X_0$ and $\phi':X'\to X_0$, the latter being called the {\it flip} of the former.
\end{definition}
In Definition \ref{def:toricflip}, the fans $\Sigma_\pm$ are those associated to the flipping cone $\sigma_0$ as described above. Note that the definition, together with Lemma \ref{lem:localtoricflip}, implies that $X$ and $X'$ are birational and isomorphic in codimension $2$, and that $\Sigma_0 \setminus (\sigma_{J_-\cup J_+})$ is simplicial. The situation of Definition \ref{def:toricflip} can be better summarized in the following commutative diagram :
 $$
 \begin{tikzcd}
 \label{diag:toricflip}
  X   & & &  & X' \arrow[ddll, rightarrow,"{\phi'}"']\arrow[llll,dashleftarrow, "{\psi}"'] \\
   & & & & \\
 V(\sigma_{J_+})\arrow[uu,rightarrow, "{\iota_+}"']   & & X_0 \arrow[uull, leftarrow, "{\phi}"'] & & V(\sigma_{J_-})\arrow[ddll, rightarrow,"{\phi_-}"']\arrow[uu,rightarrow, "{\iota_-}"'] \\ 
   & & & & \\
  & & V(\sigma_{J_+\cup J_-})  \arrow[uull, leftarrow, "{\phi_+}"']\arrow[uu,rightarrow, "{\iota_0}"'] & &
\end{tikzcd}
 $$
The maps $\phi$ and $\phi'$ are the toric morphisms induced by the refinements $\Sigma$ and $\Sigma'$ of $\Sigma_0$, while the maps $\iota_+$, $\iota_-$ and $\iota_0$ denote inclusions. Finally,  $$\psi=(\phi')^{-1}\circ\phi : X \dashrightarrow X'$$ is the birational morphism that is an isomorphism away from    $V(\sigma_{J_\pm})$. From now on, we fix $X$ and $X'$ two simplicial toric varieties related by a toric flip $\psi$, and retain the notation from the previous diagram.
\subsection{The exceptional loci and relatively ample divisors}
\label{sec:exceptionalloci}
We will be interested in stability later on, whose definition requires that the varieties be polarised. We hence turn to the description of the fibers of $\phi_\pm$, and describe some $\phi$-ample (and $\phi'$-ample) divisors. For a cone $\sigma\subset N_\R$, we denote $N_\sigma\subset N$ the sublattice
$$
N_\sigma=\mathrm{Span}(\sigma)\cap N,
$$
and the quotient map 
$$\pi_\sigma : N \to N(\sigma):=N/N_\sigma.$$
Recall that from the orbit-cone correspondence (see \cite[Section 3.2, Proposition 3.2.7]{CLS}), the toric variety $V(\sigma)$ can be obtained as the toric variety associated to the fan of cones in $(N/N_{\sigma})_\R$ :
$$
\mathrm{Star}(\sigma)=\lbrace \pi_\sigma(\tau)\:\vert\: \sigma \preceq \tau \rbrace.
$$
 In particular, the lattices of one parameter subgroups of $V(\sigma_{J_-})$ and $V(\sigma_{J_-\cup J_+})$ are respectively $N/N_{\sigma_{J_-}}$ and $N/N_{\sigma_{J_-\cup J_+}}$. One can show that the projection map 
$$
 N/N_{\sigma_{J_-}}\to N/N_{\sigma_{J_-\cup J_+}}
$$
is compatible with the fans $\mathrm{Star}(\sigma_{J_-})$ and $\mathrm{Star}(\sigma_{J_-\cup J_+})$ and induces the toric morphism 
$$
\phi_- : V(\sigma_{J_-}) \to V(\sigma_{J_-\cup J_+}).
$$
The lattices fit naturally in the sequence
$$
0\to N_{\sigma_{J_-\cup J_+}}/N_{\sigma_{J_-}}  \to N/N_{\sigma_{J_-}}\to N/N_{\sigma_{J_-\cup J_+}} \to 0.
$$
As we are interested in the fibers of $\phi_-$, we introduce the quotient lattice
$$
N_\cR:=N_{\sigma_{J_-\cup J_+}}/N_{\sigma_{J_-}},
$$
and denote the projection $N_{\sigma_{J_-\cup J_+}}\to N_\cR$, and its $\R$-linear extension, by $u\mapsto \overline{u}$. We finally introduce the fan
$$
\Sigma_\cR:=\lbrace \overline{\sigma}_J\:\vert\: J \subsetneq J_+ \rbrace,
$$
and the associated toric variety $X_\cR$.
\begin{remark}
 We keep the notation $X_\cR$ to be consistent with \cite{CLS}, where the $\cR$ stands for an {\it extremal ray} being responsible for the flip in the context of toric MMP.
\end{remark}
 For the following, see \cite[Lemma 15.4.2 and Proposition 15.4.5.(c)]{CLS}.
\begin{proposition}
 \label{prop:fibersflip}
 The fibers of $\phi_-$ are isomorphic to the $\Q$-Fano toric variety $X_\cR$. Moreover, $X_\cR$ has dimension $\vert J_+\vert-1$ and Picard rank one.
\end{proposition}
From the above, we deduce that the anticanonical divisor of $X_\cR$
$$
-K_\cR=\sum_{\rho\in \Sigma_\cR(1)} D_\rho
$$
is $\Q$-Cartier and ample. Note that by construction it can be written
$$
-K_\cR=\sum_{i\in J_+} D_{\overline{\R_+\cdot \nu_i}}=\sum_{i\in J_+} D_{\overline{\rho}_i}
$$
where we set 
$$\rho_i=\R_+\cdot \nu_i,$$
and $D_{\rho}$ stands for the torus-invariant divisor associated to the ray $\rho$. An easy exercise, using the orbit-cone correspondence, together with \cite[Proposition 15.5.1]{CLS}, shows that 
\begin{proposition}
 \label{prop:phiampledivisors}
 The $\Q$-Cartier torus invariant divisor 
 \begin{equation}
  \label{eq:phiampledivisorI}
  -D_{J_+}:= -\sum_{i\in J_+} D_{\rho_i} \in \mathrm{Div}(X)
 \end{equation}
is $\phi$-ample, while the $\Q$-Cartier invariant divisor
 \begin{equation}
  \label{eq:phiampledivisorII}
  D_{J_+}':=\psi_*(D_{J_+})= \sum_{i\in J_+} D_{\rho_i}' \in \mathrm{Div}(X')
 \end{equation}
is $\phi'$-ample.
\end{proposition}
\begin{remark}
\label{rem:notationprime}
In the above proposition, we use a superscript to $D'$ to indicate that the torus invariant divisor $D'$ is taken as the orbit closure of some ray {\it in $X'$}. As $\psi : X \dashrightarrow X'$ is an isomorphism in codimension $2$, it induces an isomorphism $\psi_*$ between the groups of torus invariant $\Q$-Cartier divisors on $X$ and $X'$, which can be written on basis elements $D_\rho \mapsto D_\rho'$ for any $\rho\in \Sigma(1)=\Sigma'(1)$. Then, to ease notation later on, we will omit the superscript to $D_\rho'$, the context being clear enough whether $D_\rho$ is considered as a divisor on $X$ or $X'$. We will also simply denote $D_{J_+}$ and $D_{J_+}'$ by $$D_+=\sum_{i\in J_+} D_{\rho_i},$$
so that the conclusion of Proposition \ref{prop:phiampledivisors} is that $-D_+$ is $\phi$-ample on $X$ and $D_+$ is $\phi'$-ample on $X'$. 

Finally, we note that we could have considered the divisor 
$$
D_-=\sum_{i\in J_-} D_{\rho_i},
$$
which is $\phi$-ample (while $-D_-$ is $\phi'$-ample).  However, the {\it wall relation}
$$
\sum_{i\in J_-} b_i \nu_i + \sum_{i\in J_+} b_i \nu_i = 0
$$
from Definition \ref{def:flippingcone} implies that from the intersection theory point of view, computing slopes with $D_+$ or $D_-$ will produce the same results regarding stability notions (see e.g. \cite[Section 6.4, Proposition 6.4.4]{CLS}). 
\end{remark}

\section{The flip functor}
\label{sec:flipfunctor}
\subsection{Equivariant sheaves and Klyachko's equivalence}
\label{sec:equivsheaves}
We now turn  to the description of the flip functor, which requires first introducing the categories of torus equivariant reflexive sheaves. For a given toric variety $X$ with fan $\Sigma$, a {\it torus equivariant reflexive sheaf} is a reflexive sheaf $\cE$ on $X$ together with an isomorphism
$$\varphi : \alpha^*\cE \to \pi_2^*\cE$$
satisfying certain cocyle conditions,
where $\alpha : T \times X \to X$ and 
$\pi_2 : T \times X \to X$  stand for the torus action and the projection on $X$ respectively (see for example \cite[Section 5]{Per04}). 
\begin{definition}
 A {\it toric sheaf} is a torus equivariant reflexive sheaf.
\end{definition}
Klyachko has shown  (see \cite{Kly90} for locally free sheaves and \cite{Per04} in general) that any toric sheaf is uniquely described by a {\it family of filtrations}, denoted  $$(E,E^\rho(i))_{\rho\in\Sigma(1),i\in\Z}.$$
Here, $E$ stands for a finite dimensional complex vector space of dimension $\rank(\cE)$, and for each ray $\rho\in\Sigma(1)$, $(E^\rho(i))_{i\in\Z}$ is a bounded {\it increasing} filtration of $E$ (we will use increasing filtrations as in \cite{Per04}, rather than decreasing ones as in \cite{Kly90}). Then, the equivariant reflexive sheaf $\cE$ is recovered from the formula, for $\sigma\in\Sigma$ :
\begin{equation*}
  \label{eq:sheaffromfiltrations}
  \Gamma(U_{\sigma}, \cE):=\bigoplus_{m\in M} \bigcap_{\rho\in\sigma(1)} E^\rho(\langle m,u_\rho\rangle)\otimes \chi^m
 \end{equation*}
 where $u_\rho\in N$ is the primitive generator of $\rho$, $\langle \cdot,\cdot\rangle$ the duality pairing and $\chi^m$ the weight $m$ character function.

\begin{remark}
If $(E, E^\rho(i))$ and $(F, F^\rho(i))$ denote respectively the families of increasing and
decreasing filtrations of a toric sheaf $\cE$, then there are related by the formula
$$
F^\rho(i) = E^\rho(-i).
$$
In our study of stability, nothing changes in the choice of increasing or decreasing
filtrations.
\end{remark}
 
 A morphism between two families of filtrations 
 $$ b : (E_1,E_1^\rho(i))_{\rho\in\Sigma(1),i\in\Z}\to(E_2,E_2^\rho(i))_{\rho\in\Sigma(1),i\in\Z}$$
 is a linear map $b : E_1 \to E_2$ that satisfies $$b(E_1^\rho(i))\subset E_2^\rho(i)$$ for any $\rho\in\Sigma(1)$ and $i\in\Z$. Any such morphism corresponds uniquely to a morphism between the associated reflexive sheaves, and a fundamental result of Klyachko and Perling (\cite{Kly90,Per04}) asserts that the categories of families of filtrations and of toric sheaves are equivalent. For our purposes, it seems more natural to slightly restrict the definition of morphisms. We will consider morphisms between toric sheaves $\cE_i$ to be equivariant morphisms of coherent sheaves $\beta : \cE_1 \to \cE_2$ that satisfy that $\mathrm{Im}(\beta)$ is a saturated reflexive subsheaf of $\cE_2$. Those morphisms correspond through Klyachko's equivalence to linear maps $b :E_1 \to E_2$ such that 
 $$b(E_1^\rho(i))=b(E_1)\cap E_2^\rho(i)$$ 
 for any $(\rho,i)$, as can be seen via \cite[Lemma 2.15]{NapTip} for example. We denote by $\Ref^T(X)$ on one hand and by $\Filt(X)$ on the other, the categories of toric sheaves and of families of filtrations,  endowed with those classes of morphisms. We will denote by 
 $$
 \Kl : \Filt(X) \to  \Ref^T(X)
 $$
 Klyachko's functor as described above. Then, Klyachko and  Perling's work readily implies the following :
 \begin{theorem}[\cite{Kly90,Per04}]
  The functor $\Kl$ is an equivalence of categories.
 \end{theorem}
 
 \begin{remark}
  A nice feature of the categories $\Filt(X)$ and $\Ref^T(X)$ is that they are abelian. This is no longer true when we consider all morphisms of reflexive sheaves, as for example the quotient of $\cO_{\P^1}$ by the subsheaf $\cO_{\P^1}(1)$ is torsion, and hence not reflexive.
 \end{remark}

 \subsection{Flip functor}
 \label{sec:functordefinition}
 Assume now that $\phi' : X' \to X_0$ is the flip of $\phi : X \to X_0$ as in the previous section. As $\phi$ and $\phi'$  are isomorphisms in codimension $2$, we have $$\Sigma(1)=\Sigma_0(1)=\Sigma'(1).$$ We deduce that there are equivariant injections $i$ (resp. $i_0$ and $i'$) of the $T$-invariant Zariski open set 
 $$
 U:=\bigcup_{\tau\in\Sigma(0)\cup\Sigma(1)} U_\tau
 $$
 into $X$ (resp. $X_0$ and $X'$). Then, from \cite[Proposition 1.6]{Har80}, we deduce that for any toric sheaf $\cE\in \mathrm{Obj}(\Ref^T(X))$ (resp. $\cF\in \mathrm{Obj}(\Ref^T(X_0))$ and $\cG\in \mathrm{Obj}(\Ref^T(X'))$), we have 
 $$
 i_*(\cE_{\vert U})\simeq \cE
 $$
 (resp.  $(i_0)_*(\cF_{\vert U})\simeq \cF$ and $i'_*(\cG_{\vert U})\simeq \cG$). As reflexive sheaves are normal, meaning their sections extend over codimenion-$2$ Zariski-closed subsets, we have :
 \begin{proposition}
  \label{prop:equivalentcategoriesflip}
  The pushforward $i_*$ (resp. $i'_*$ and $(i_0)_*$) induces an equivalence of categories 
  $$
  i_* : \Ref^T(U)\to \Ref^T(X)
  $$
  (resp. $\Ref^T(U)\simeq \Ref^T(X_0)$ and $\Ref^T(U)\simeq \Ref^T(X')$). Hence, we have equivalences 
  $$
  \Ref^T(X)\simeq \Ref^T(X_0)\simeq \Ref^T(X').
  $$
 \end{proposition}
 It is straightforward to check that the equivalence
 $$
 \Ref^T(X)\simeq \Ref^T(X_0)
 $$
 is induced by the pushforward $\phi_*$ while the equivalence 
 $$
 \Ref^T(X')\simeq \Ref^T(X_0)
 $$
 is induced by $\phi_*'$. Moreover, the categories of families of filtrations on $U$, $X$, $X'$ and $X_0$ are readily the same, and the above equivalence of categories simply correspond to the self-equivalence of $\Filt(X)$ induced by the identity on objects and morphisms.
 \begin{definition}
  \label{def:flipfunctor}
  We define the {\it flip functor} 
  $$
  \psi_* :  \Ref^T(X) \to \Ref^T(X')
  $$
  to be the composition of functors induced by $(\phi'_*)^{-1}$ and $\phi_*$.
 \end{definition}
We conclude this section by noting that the flip functor sends the tangent sheaf $\cT_X$ of $X$ to the tangent sheaf $\cT_{X'}$ of $X'$:
\begin{lemma}
 \label{lem:tangentpreserved}
 We have $$\phi_* \cT_X= \cT_{X_0}=\phi'_*\cT_{X'}.$$
\end{lemma}

\begin{proof}
 It follows from the facts that $X$, $X_0$ and $X'$ are normal, and their tangent sheaves reflexive. Hence, the sheaves are determined by their restriction to the complement of the exceptional loci of $\phi$ and $\phi'$ which have codimension greater or equal to $2$ (see \cite[Proposition 1.6]{Har80}).
\end{proof}

 \subsection{The logarithmic subcategories}
 \label{sec:logcase}
 In Section \ref{sec:secondtheorem}, we will be interested in specific subcategories of $\Ref^T(X)$ and $\Ref^T(X')$. For any $\Delta\subset \Sigma(1)$, we introduce the torus invariant divisor
 $$
 D_\Delta:=\sum_{\rho\in \Delta} D_\rho
 $$
 and the full subcategory $\Ref^T(X,D_\Delta)$ of $\Ref^T(X)$ whose objects are the toric sheaves on $X$ whose associated families of filtrations $(E,E^\rho(i))_{\rho\in\Sigma(1), i\in\Z}$ satisfy for all $\rho\in\Delta$, there is $a_\rho \in \Z$ such that:
 \begin{equation}
  \label{eq:definingconditionlogcategory}
  E^\rho(i)=\left\{ 
  \begin{array}{ccc}
   0 & \mathrm{ if } & i< a_\rho \\
   E & \mathrm{ if } & i\geq a_\rho .
  \end{array}
  \right.
 \end{equation}

\begin{remark}
From \cite[Theorem 1.1]{Nap-logtangent}, the logarithmic tangent sheaf $\cT_X(-\log D_\Delta)$ belongs to $\mathrm{Obj}(\Ref^T(X,D_\Delta))$, which justifies our choice of terminology.
\end{remark}
It is then straightforward to see that the flip functor $\psi_*$ induces an equivalence between $\Ref^T(X,D_\Delta)$ and $\Ref^T(X',D_\Delta')$, where we use $D_\Delta'$ to denote $\sum_{\rho\in\Delta} D_\rho'$. Note also that $\psi_*$ sends $\cT_X(-\log D_\Delta)$ to $\cT_{X'}(-\log D_\Delta')$.

\section{Flips and stability for a given sheaf}
\label{sec:specificfamilies}

\subsection{Slope stability of toric sheaves}
\label{sec:slopestability}
Let $(X,L)$ be a polarised complex variety. Recall that a reflexive sheaf $\cE$ on $(X,L)$ is said to be {\it slope stable} (resp. {\it slope semistable}), or simply stable (resp. semistable) for short, if for any coherent and saturated subsheaf $\cF\subset \cE$ of strictly smaller rank, one has
$$
\mu_L(\cF)<\mu_L(\cE),
$$
(resp. $\mu_L(\cF)\leq \mu_L(\cE)$),
where for any coherent torsion-free sheaf $\cF$, the {\it slope} $\mu_L(\cF)$ is defined by
$$
\mu_L(\cF)=\frac{c_1(\cF)\cdot L^{n-1}}{\rank(\cF)}\in\Q.
$$
A {\it polystable sheaf} is a direct sum of stable ones with the same slope. A sheaf will be called {\it unstable} if it is {\it not semistable}.
\begin{remark}
When referring to a specific polarisation $L$ used to define stability notions, we will use the terminology $L$-stable (resp. $L$-unstable, $L$-semistable, etc).
\end{remark}

A remarkable fact, proved by Kool (\cite[Proposition 4.13]{Koo11}), is that if we assume $X$ and $\cE$ to be toric, to check stability for $\cE$, it is enough to compare slopes with equivariant and saturated reflexive subsheaves, that is sub objects of $\cE$ in $\Ref^T(X)$ (note that this was proved in the smooth case by Kool, but it was noted in \cite{CT22} that the proof extends in the normal case). If  $(E,E^\rho(\bullet))_{\rho\in\Sigma(1)}$ stands for the family of filtrations of $\cE$, any saturated equivariant reflexive subsheaf of $\cE$ is associated to a family of filtrations of the form $(F, F\cap E^\rho(i))_{\rho\in\Sigma(1),i\in\Z}$ for some vector subspace $F\subsetneq E$ (see \cite[Lemma 2.15]{NapTip}). Moreover, Klyachko's formula for the slope of a toric sheaf (\cite[Corollary 2.18]{CT22}) is 
\begin{equation}
 \label{eq:slope}
 \mu_L(\cE)=-\frac{1}{\rank(\cE)}\sum_{\rho\in\Sigma(1)} \iota_\rho(\cE)\, \deg_L(D_\rho),
\end{equation}
 where $\deg_L(D_\rho)$ is the degree with respect to $L$, and
 $$
 \iota_\rho(\cE):=\sum_{i\in\Z} i \left(\dim(E^\rho(i))-\dim (E^\rho(i-1))\right).
 $$
 To sum up, we have
\begin{proposition}
\label{prop:stabilityequivariantcase}
The toric sheaf associated to $(E,E^\rho(i))_{\rho\in\Sigma(1), i\in\Z}$ is stable  if and only if for any subspace $F\subsetneq E$, we have 
$$
\frac{1}{\dim(F)}\sum_{\rho\in\Sigma(1)} \iota_\rho(F)\, \deg_L(D_\rho) > \frac{1}{\dim(E)}\sum_{\rho\in\Sigma(1)} \iota_\rho(\cE)\, \deg_L(D_\rho),
$$
where 
$$
 \iota_\rho(F):=\displaystyle\sum_{i\in\Z} i \left(\dim(F\cap E^\rho(i))-\dim (F\cap E^\rho(i-1))\right).
$$
The similar statement holds for semistability by replacing the strict inequality by a weak inequality.
 \end{proposition}
 \begin{remark}
  As observed in Section \ref{sec:functordefinition}, for a given flip as in Section \ref{sec:flips}, the families of filtrations for $\cE\in\Obj(\Ref^T(X))$, $\psi_*\cE$ and $\phi_*\cE$ are the same. We thus have the equalities $\iota_\rho(\cE)=\iota_\rho(\phi_*\cE)=\iota_\rho(\psi_*\cE)$. Then, to compare slopes on $X$, $X_0$ and $X'$, only the terms coming from the degrees of the invariant divisors $D_\rho$'s will vary according to the polarisations on each variety.
 \end{remark}

\subsection{Main result and its proof}
\label{sec:main}
Consider now $\phi : X \to X_0$ and its toric flip $\phi' : X' \to X_0$ as defined in Section \ref{sec:toricflips}. From Proposition \ref{prop:phiampledivisors} (recall also Remark \ref{rem:notationprime}), for any ample Cartier divisor $L_0$ on $X_0$, there exists $\ep_0 >0$ such that the divisors
$$
L_{-\ep}:=\phi^*L_0 - \ep D_+
$$
on $X$ and 
$$
L_{\ep}:=(\phi')^*L_0 + \ep D_+
$$
on $X'$ define ample $\Q$-Cartier divisors for $\ep\in (0,\ep_0)$. We will then be interested in the behaviour of stability for toric sheaves related by the flip functor on $(X, L_{-\ep})$ and $(X', L_\ep)$, for $0<\ep<\ep_0$. Note that a necessary condition for stability of an element $\cE\in\Obj(\Ref^T(X))$ under those polarisations is $L_0$-semistability of $\phi_* \cE$.

Conversely, if the sheaf $\phi_* \cE$ is $L_0$-semistable, it then admits a Jordan-H\"older filtration
$$
0=\cE_1 \subseteq \cE_2 \subseteq \ldots \subseteq \cE_\ell= \phi_* \cE
$$
by slope semistable coherent and saturated subsheaves with stable quotients of the same slope as $\phi_* \cE$ (see e.g. \cite{HuLe}).
We denote by
$$
\Gr(\phi_* \cE):=\bigoplus_{i=1}^{\ell-1}\cE_{i+1}/\cE_i
$$
the graded object of $\phi_* \cE$ and by $\fE$ the finite set of equivariant and saturated reflexive subsheaves $\cF \subseteq \phi_* \cE$ arising in a Jordan-H\"older filtration of $\phi_* \cE$.
Note that by Proposition \ref{prop:equivalentcategoriesflip}, for any $\cF \in \fE$,
$(\phi^* \cF)^{\vee \vee}$ (resp. $((\phi')^* \cF)^{\vee \vee}$) is saturated in $\cE$ (resp. in $\psi_* \cE$).

\begin{theorem}
 \label{theo:main}
 Let $\cE$ be a toric sheaf on $X$. Then, up to shrinking $\ep_0$, we have for all $\ep\in(0,\ep_0)$ :
 \begin{enumerate}[label=(\roman*)]
  \item If $\phi_*\cE$ is $L_0$-stable, then $\cE$ (resp. $\psi_*\cE$) is $L_{-\ep}$-stable on $X$ (resp. $L_\ep$-stable on $X'$).
  \item If $\phi_*\cE$ is $L_0$-unstable, then $\cE$ (resp. $\psi_*\cE$) is $L_{-\ep}$-unstable on $X$ (resp. $L_\ep$-unstable on $X'$).
  \item If $\phi_*\cE$ is $L_0$-semistable, and if for every $\cF \in \fE$,
  $$
  \dfrac{c_1(\cE) \cdot D_{+} \cdot (\phi^* L_0)^{n-2}}{\rk(\cE)} <
  \dfrac{c_1((\phi^*\cF)^{\vee \vee}) \cdot D_{+} \cdot (\phi^* L_0)^{n-2}}{\rk(\cF)}
  $$
  then $\cE$ (resp. $\psi_*\cE$) is $L_{-\ep}$-stable on $X$ (resp. $L_\ep$-unstable on $X'$).
  \label{point:theo-main3}
  \item If $\phi_*\cE$ is $L_0$-semistable, and if for every $\cF \in \fE$,
  $$
  \dfrac{c_1(\cE) \cdot D_{+} \cdot (\phi^* L_0)^{n-2}}{\rk(\cE)} >
  \dfrac{c_1((\phi^*\cF)^{\vee \vee}) \cdot D_{+} \cdot (\phi^* L_0)^{n-2}}{\rk(\cF)}
  $$
  then $\cE$ (resp. $\psi_*\cE$) is $L_{-\ep}$-unstable on $X$ (resp. $L_\ep$-stable on $X'$).
  \label{point:theo-main4}
 \end{enumerate}
\end{theorem}

\begin{remark}
 Note that, given the semistability of $\phi_*\cE$, the numerical criterion in points $(iii)$ and $(iv)$ only requires testing a {\it finite} number of inequalities, as $\fE$ is finite. This makes this criterion useful in practice.
\end{remark}

Before proving this theorem, we first recall some facts on intersection products in toric varieties that will be used. 
Let $\{u_1, \ldots, u_k\}$ be a set of primitive elements of $N$ such that $\sigma = \mathrm{Cone}(u_1, \ldots, u_k)$ is simplicial. We define $\mult(\sigma)$ as the index of the sublattice $\Z u_1 + \ldots + \Z u_k$ in $N_{\sigma} = \Span(\sigma) \cap N$.
If $\Sigma$ is simplicial, according to \cite[Section 5.1]{ful93}, one can consider intersections of cycles or cycle classes with {\it rational} coefficients. The Chow group
$$
A^{\bullet}(X)_\Q = \bigoplus_{p=0}^{n} A^{p}(X) \otimes \Q =
\bigoplus_{p=0}^{n} A_{n-p}(X) \otimes \Q
$$
has the structure of a graded $\Q$-algebra and by \cite[Lemma 12.5.2]{CLS}, if $\rho_1, \ldots, \rho_d \in \Sigma(1)$ are distinct, then in $A^{\bullet}(X)_\Q$, we have
\begin{equation}\label{eq:intersection-toric-divisors}
[D_{\rho_1}] \cdot [D_{\rho_2}] \cdots [D_{\rho_d}] = \left\lbrace
\begin{array}{ll}
\dfrac{1}{\mult(\sigma)} [ V(\sigma)] &
\text{if $\sigma = \rho_1 + \ldots + \rho_d \in \Sigma$} \\ 0 & \text{otherwise.}
\end{array}
\right..
\end{equation}
If $\chi^m$ is the weight $m$ character function on $X_\Sigma$, then by \cite[Proposition 4.1.2 and (12.5.4)]{CLS}, the divisor of $\chi^m$ is given by
\begin{equation}\label{eq:divisor-of-character}
{\rm div}(\chi^m) = \sum_{\rho \in \Sigma(1)} \<m, u_\rho \> D_\rho
\end{equation}
and ${\rm div}(\chi^m) \sim_{\rm lin} 0$ in $A^1(X_\Sigma)$.

\begin{proof}[Proof of Theorem \ref{theo:main}]
We first prove that for any $\rho \in \Sigma(1)$,
\begin{equation}\label{eq:intersection-withD+}
\phi_*(D_\rho \cdot D_+) = \phi'_*(D'_\rho \cdot D'_+).
\end{equation}
We use the notation of Section \ref{sec:flips} for the toric flip, and set  
$$\Delta = \{\Cone(\nu_i) : i \in J_+ \}.$$
Recall from Definition \ref{def:flippingcone} that $J_+$ and $J_-$ have at least two elements. It then follows from the definition of $\Sigma_\pm$ that for any $\rho \in \sigma_0(1) \setminus \Delta$, and any $j \in J_+$, $\rho + \Cone(\nu_j)$ is a two-dimensional cone of $\Sigma_0$, $\Sigma_+$ and $\Sigma_{-}$.
Therefore, we deduce that in the Chow ring $A^{\bullet}(X_0)_{\Q}$,
$$
\phi_*(D_\rho \cdot D_{\rho_j}) = D_\rho \cdot D_{\rho_j}
\quad \text{and} \quad
\phi'_*(D'_\rho \cdot D'_{\rho_j}) = D_\rho \cdot D_{\rho_j}.
$$
If $\rho \in \Sigma(1) \setminus \sigma_0(1)$, then for any $j \in J_+$, 
$$\rho + \Cone(\nu_j) \notin \{\tau : \tau \preceq \sigma_0 \}.
$$
As by Definition \ref{def:toricflip},
$$\Sigma_{\vert N_\R \setminus \sigma_0}= \Sigma_{\vert N_\R \setminus \sigma_0}'=( \Sigma_0)_{\vert N_\R\setminus \sigma_0},$$
we deduce that:
\begin{itemize}
\item
either $\rho + \Cone(\nu_j) \in \Sigma_0 \setminus \{ \tau : \tau \preceq \sigma_0\}$, and then in $A^{\bullet}(X_0)_{\Q}$, we have
$$
\phi_*(D_\rho \cdot D_{\rho_j}) = D_\rho \cdot D_{\rho_j} =
\phi'_*(D'_\rho \cdot D'_{\rho_j});
$$
\item
or
$\rho + \Cone(\nu_j) \notin \Sigma_0 \setminus \{ \tau : \tau \preceq \sigma_0\}$, in which case $D_\rho \cdot D_{\rho_j} = 0$ in $A^{\bullet}(X)_\Q$ and $D'_\rho \cdot D'_{\rho_j} = 0$ in $A^{\bullet}(X')_\Q$.
\end{itemize}
This proves that for any $\rho \in \Sigma(1) \setminus \Delta$ and any $j \in J_+$,
$$
\phi_*(D_\rho \cdot D_{\rho_j}) = \phi'_*(D'_\rho \cdot D'_{\rho_j})
$$
and then by linearity and the definition of $D_+$:
$$
\phi_*(D_\rho \cdot D_+) = \phi'_*(D'_\rho \cdot D'_+).
$$
We now assume that $\rho \in \Delta$. By Lemma \ref{lem:localtoricflip}, one has $\dim V(\sigma_{J_+}) = n - |J_+|$; therefore $\sigma_{J_+}$ is a simplicial cone and then $\{\nu_j : j \in J_+\}$ forms part of a $\Q$-basis of $N \otimes_\Z \Q$.
Let $\{\nu_j^* : j \in J_+ \}$ be a part of a $\Q$-basis of $M \otimes_{\Z} \Q$ such that for any $i, j \in J_+$,
$$
\< \nu_j^*, \nu_i \> = \left\lbrace
\begin{array}{ll}
0 & \text{if $i \neq j$} \\ 1 & \text{if $i=j$}
\end{array}
\right..
$$
For any $j \in J_+$, there is $a_j \in \N^*$ such that $a_j \nu_j^* \in M$. By using (\ref{eq:divisor-of-character}) with $m = a_j \nu_j^*$, we get
$$
a_j D_{\rho_j} \sim_{\rm lin} - \sum_{\rho \in \Sigma(1) \setminus \Delta}\< a_j \nu_j^*, u_\rho \> D_\rho .
$$
on $X$ and $X'$. By the first cases, we deduce that for any $j \in J_+$,
\begin{align*}
\phi_*(D_{\rho_j} \cdot D_+) & = - \phi_* \left(
\sum_{\rho \in \Sigma(1) \setminus \Delta} \< \nu_j^*, u_\rho \> D_\rho \cdot D_+
\right)
\\ & =
- \phi'_* \left(
\sum_{\rho \in \Sigma(1) \setminus \Delta} \< \nu_j^*, u_\rho \> D'_\rho \cdot D'_+
\right)
\\ & = \phi'_*( D'_{\rho_j} \cdot D'_+ ).
\end{align*}
This concludes the proof of (\ref{eq:intersection-withD+}).

We can now compute the slopes.
By the Projection formula \cite[Proposition 2.3]{Fulton-intersection}, for any $\rho \in \Sigma(1)$, one has
\begin{align*}
& D_\rho \cdot (\phi^* L_0)^{n-1} = \deg_{L_0}(D_\rho),
\\
& D'_\rho \cdot ((\phi')^* L_0)^{n-1} = \deg_{L_0}(D_\rho)
\quad \text{and}\\
& D_\rho \cdot D_+ \cdot (\phi^* L_0)^{n-2} = D'_\rho \cdot D'_+ \cdot ((\phi')^* L_0)^{n-2}.
\end{align*}
As we have
\begin{align*}
& (L_{-\ep})^{n-1} = (\phi^* L_0)^{n-1} - (n-1) \ep D_{+} \cdot (\phi^* L_0)^{n-2} + O(\ep^2)
\quad \text{and} \\
& (L_{\ep})^{n-1} = ((\phi')^* L_0)^{n-1} + (n-1) \ep D_{+} \cdot ((\phi')^* L_0)^{n-2} + O(\ep^2),
\end{align*}
we deduce that: for any coherent sheaf $\cE$ on $X$,
$$
\mu_{L_{-\ep}}(\cE) = \mu_{L_0}(\phi_* \cE) -
\dfrac{c_1(\cE) \cdot D_+ \cdot (\phi^* L_0)^{n-2} }{\rank(\cE)}(n-1) \ep + O(\ep^2)
$$
and
$$
\mu_{L_{\ep}}(\psi_* \cE) = \mu_{L_0}(\phi_* \cE) +
\dfrac{c_1(\cE) \cdot D_+ \cdot (\phi^* L_0)^{n-2} }{\rank(\cE)}(n-1) \ep + O(\ep^2).
$$
Assume now $\cE\in\Obj(\Ref^T(X))$ is given by the family of filtrations $(E, E^\rho(j))$. By Proposition \ref{prop:stabilityequivariantcase}, to check the stability of $\cE$, it is enough to compare the slope of $\cE$ with the slope of any equivariant reflexive sheaf $\cF$ given by the family of filtrations $(F, F \cap E^\rho(j))$ for $F \subsetneq E$ a subspace.
As the set of vector subspaces $F\subset E$ on which it is necessary to test slopes is actually finite (see \cite[Lemma 2.17]{NapTip}), we deduce from the above $\ep$-expansions for the slopes that there is $\ep_0 >0$ such that for all $\ep \in (0, \ep_0)$:
\begin{itemize}
\item
if $\phi_*\cE$ is $L_0$-stable, then $\cE$ (resp. $\psi_*\cE$) is $L_{-\ep}$-stable on $X$ (resp. $L_\ep$-stable on $X'$);
\item
and if $\phi_*\cE$ is $L_0$-unstable, then $\cE$ (resp. $\psi_*\cE$) is $L_{-\ep}$-unstable on $X$ (resp. $L_\ep$-unstable on $X'$).
\end{itemize}
We now consider the case where $\phi_* \cE$ is $L_0$-semistable.
We first observe, as in the (un)stable case, that there is $\ep_1 >0$, such that for all $\ep \in (0, \ep_1)$ and for all $\cF\in\Obj(\Ref^T(X_0))$ with $\cF \subsetneq \phi_* \cE$ such that $\mu_{L_0}(\cF)< \mu_{L_0}(\phi_* \cE)$ one has
$$
\mu_{L_{-\ep}}( (\phi^* \cF)^{\vee \vee} ) < \mu_{L_{-\ep}}(\cE) \quad \text{and} \quad
\mu_{L_{\ep}}( ((\phi')^* \cF)^{\vee \vee} ) < \mu_{L_{\ep}}( \psi_* \cE).
$$
If $\cF \subsetneq \phi_* \cE$ is a sub object such that $\mu_{L_0}(\cF) = \mu_{L_0}(\phi_*\cE)$, then there is a Jordan-H{\"o}lder filtration
$$
0 = \cE_1 \subseteq \ldots \subseteq \cE_l = \phi_*\cE
$$
with $l \geq 1$ such that $\cF= \cE_i$ for some $i \in \{1, \ldots, l \}$ (see \cite[Section 1.6]{HuLe}) and we deduce that $\cF \in \fE$. Therefore, from the expansions of the slopes, to get the points $(iii)$ and $(iv)$ of the theorem, it suffices to compare
$$
\dfrac{c_1(\cE) \cdot D_+ \cdot (\phi^* L_0)^{n-2} }{\rank(\cE)} \quad \text{and} \quad
\dfrac{c_1((\phi^* \cF)^{\vee \vee}) \cdot D_+ \cdot (\phi^* L_0)^{n-2} }{\rank(\cF)}
$$
for any $\cF \in \fE$. By uniqueness of the reflexive hull of the graded object of a Jordan--H\"older filtration (\cite[Theorem 1.6.7]{HuLe}),  $\fE$ is finite, and the result follows.
\end{proof}

\begin{remark}
In the proof, we have shown that for any $\rho \in \Sigma(1)$, 
$$
D_\rho \cdot D_+ \cdot (\phi^* L_0)^{n-2}= D'_\rho \cdot D'_+ \cdot ((\phi')^* L_0)^{n-2}.
$$
In Equation (\ref{eq:degree-divisorX}), the coefficient of $\ep^2$ in the $\ep$-expansion of $\deg_{L_{-\ep}}(D_\rho)$ corresponds to $D_\rho \cdot (D_{+})^2$.
By (\ref{eq:degree-divisorX}) and (\ref{eq:degree-divisorX2}) we note that there exists $\rho \in \Sigma(1)$ such that
$$
D_\rho \cdot (D_{+})^2 \neq D'_\rho \cdot (D'_{+})^2.
$$
Therefore, if $\ell \geq 2$, for any $\rho \in \Sigma(1)$, we do not necessarily have the equality
$$
D_\rho \cdot (D_+)^{\ell} \cdot (\phi^* L_0)^{n-1- \ell}=
D'_\rho \cdot (D'_+)^{\ell} \cdot ((\phi')^* L_0)^{n-1- \ell}.
$$
The arguments used to prove Theorem \ref{theo:main} are very close to those used in \cite{NapTip}. One should be careful though that the results from \cite{NapTip} do not directly imply Theorem \ref{theo:main}, as $X_0$ is not $\Q$-factorial.
\end{remark}

\begin{remark}
 While the case when $\phi_*\cE$ is semistable on $X_0$ is not fully covered by Theorem \ref{theo:main}, items $(iii)$ and $(iv)$, one can easily adapt the numerical criterion of \cite[Theorem 1.3]{NapTip} to take into account higher order terms in the $\ep$-expansions of the $L_{-\ep}$ and $L_\ep$ slopes, and obtain a full description of the stability behaviour of $\cE$ in terms of that of $\phi_*\cE$, for the considered polarisations.
\end{remark}

Actually, Theorem \ref{theo:main} holds for some specific flat families of toric sheaves. We recall that the {\it characteristic function} $\chi$ of an equivariant reflexive sheaf $\cF$ with family of filtrations $(F, \{F^{\rho}(j) \})$ is the function
$$
\begin{array}{cccl}
\chi(\cF) : & M & \longrightarrow & \Z^{\sharp \Sigma(n)} \\
& m & \longmapsto & \left(
\dim \left( \bigcap_{\rho \in \sigma(1)} F^{\rho}(\<m, u_\rho \>) \right)
\right)_{\sigma \in \Sigma(n)}
\end{array}.
$$
Let $S$ be a scheme of finite type over $\C$ and $\cE = (\cE_s)_{s \in S}$ be an $S$-family of equivariant reflexive sheaves over $X$ (see \cite[Section 3.5]{NapTip} for more details).
We denote by $(E_s, E_{s}^{\rho}(i) )$ the family of filtrations of $\cE_s$. There is a collection of increasing filtrations of reflexive sheaves
$$
(\cF, \{ \cF_{m}^{\rho}: m \in M \}_{\rho \in \Sigma(1)} )
$$
such that for any $s \in S$ and all $m \in M$,
$$
E_s = \cF(s) \quad \text{and} \quad E_{s}^{\rho}(\<m, u_\rho \>) = \cF_{m}^{\rho}(s)
$$
where $\cF(s)$ and $\cF_{m}^{\rho}(s)$ are respectively the fibers of $\cF$ and $\cF_{m}^{\rho}$ at $s$.

\begin{lemma}\label{lem:dim-constant}
Let $X$ be a toric variety given by a simplicial fan $\Sigma$ and let $\cE= (\cE_s)_{s \in S}$ be an $S$-family of equivariant reflexive sheaves over $X$ such that:
\begin{enumerate}
\item $\cE$ is locally free over $X \times S$, or
\item the map $s \mapsto \chi(\cE_s)$ is constant.
\end{enumerate}
Then, for all $\rho \in \Sigma(1)$ and $j \in \Z$, the map $s \mapsto \dim( E_s^{\rho}(j) )$ is constant.
\end{lemma}

\begin{proof}
If $\cE$ is locally free over $X \times S$, by \cite[Proposition 3.13]{payne2008} (Klyachko's compatibility condition for $S$-families of locally free sheaves), for any $\sigma \in \Sigma(n)$, there is a multiset $A_\sigma \subseteq M$ of size $\rk(\cE)$ such that for any $m \in M$, $\cF_{m}^{\rho}$ is a locally free sheaf of rank
$$
\left| \{ \alpha \in A_\sigma : \< \alpha, u_\rho \> \leq \<m, u_\rho \> \} \right| .
$$
As for any $s \in S$ and $m \in M$, $\dim( \cF_{m}^{\rho}(s) ) = \rk( \cF_{m}^{\rho} )$,
we deduce that the map
$$
s \longmapsto \dim(E_{s}^{\rho}(\< m, u_\rho \>))
$$
is constant.

We now assume that the map $s \mapsto \chi(\cE_s)$ is constant.
For any $\rho \in \Sigma(1)$, we denote by $i_\rho$ the smallest integer such that for any $j \geq i_\rho$ and any $s \in S$,
$$
E_s^\rho(j) = E_s.
$$
Let $\sigma \in \Sigma(n)$. The set $\lbrace u_\rho : \rho \in \sigma(1) \rbrace$ is a $\Q$-basis of $N \otimes_\Z \Q$; we denote by $\{u_\rho^*: \rho \in \sigma(1) \}$ its dual basis.
For any $\rho' \in \sigma(1)$, there is $m' \in M$, such that $\<m', u_{\rho'} \>=j$. Let
$m \in M$ be given by
$$
m = m' + \sum_{\rho \in \sigma(1) \setminus \{\rho'\}} a_\rho u^*_\rho
$$
where for any $\rho \in \sigma(1) \setminus \{\rho'\}$, $a_\rho \in \Z$ satisfies $a_\rho u^*_\rho \in M$ and $a_\rho + \<m', u_\rho \> \geq i_\rho$. By construction of $m$, one has
$$
\bigcap_{\rho \in \sigma(1)} E_s^\rho(\<m, u_\rho \>) = E_{s}^{\rho'}(j).
$$
As $s \mapsto \chi(\cE_s)$ is constant, we deduce that the map $s \mapsto \dim(E_s^{\rho'}(j))$ is constant for any $\rho' \in \sigma(1)$ and any $j \in \Z$.
\end{proof}

If $\cE$ is an $S$-family of equivariant reflexive sheaves on $X$ which satisfies the conditions of Lemma \ref{lem:dim-constant}, then by \cite[Lemma 3.12]{NapTip}, for any ample $\Q$-Cartier divisor $L$ on $X$, the set
$$
\left\lbrace
\mu_{L}(\cG_s): s \in S, \:
\text{$\cG_s$ is an equivariant and saturated reflexive subsheaf of $\cE_s$}
\right\rbrace
$$
is finite.
Therefore, in that case, the $\ep_0$ in Theorem \ref{theo:main} can be taken uniformly for $(\cE_s)_{s \in S}$.

\subsection{An example}
\label{sec:example}
We illustrate in this section Theorem \ref{theo:main} by providing an example of a tangent sheaf that goes from unstable to stable through a flip.

We denote by $(e_1, e_2, e_3)$ the standard basis of $\Z^3$. Let
$$
u_1 = e_1 , \quad u_2= e_1 + e_2 - e_3, \quad u_3 = e_2, \quad u_4 = e_3, \quad
u_0 = -(e_1+e_2+e_3)
$$
and $\Sigma_0$ be a fan in $\R^3$ given by
$$
\Sigma_0 = \{ \Cone(u_1, u_2, u_3, u_4) \} \cup \bigcup_{i=1}^{4}
\{ \Cone(A): A \subseteq \{u_0, u_i, u_{i+1} \} \}
$$
where $u_5 = u_1$. We denote by $\sigma_0$ the flipping cone $\Cone(u_1, u_2, u_3, u_4)$; we have $$u_2 + u_4 - u_1 - u_3 = 0.$$
Let
$$
\Sigma = (\Sigma_0 \setminus \{\sigma_0\}) \cup \Sigma_{+} \quad \text{and} \quad
\Sigma' = (\Sigma_0 \setminus \{\sigma_0\}) \cup \Sigma_{-}$$
where
\begin{align*}
\Sigma_{+} &= \{ \Cone(u_j: j \in J): J \subset \{1, \ldots, 4 \} \text{ and } \{1,3\} \nsubseteq J \} \quad \text{and}
\\
\Sigma_{-} &= \{ \Cone(u_j: j \in J): J \subset \{1, \ldots, 4 \} \text{ and } \{2,4\} \nsubseteq J \}.
\end{align*}
We denote by $D_i$ the torus invariant divisor associated to the ray $\Cone(u_i)$.
By using (\ref{eq:divisor-of-character}) with $m \in \{e_1, e_2, e_3\}$, we get the following linear equivalences of divisors on $X_0$, $X$ and $X'$:
$$
D_1 \sim_{\rm lin} D_3 \sim_{\rm lin} D_0 - D_2 \quad \text{and} \quad
D_4 \sim_{\rm lin} D_0 + D_2.
$$
By \cite[Theorem 4.2.8.(d)]{CLS}, the divisor $D_0$ generates the set of invariant Cartier divisors of $X_0$. As $\Sigma$ (resp. $\Sigma'$) is simplicial, by \cite[Proposition 4.2.7]{CLS} any invariant divisor of $X$ (resp. $X'$) is $\Q$-Cartier.

\begin{figure}[t]
\centering

\begin{subfigure}{0.3\textwidth}
\begin{tikzpicture}[scale=0.6]
\fill[gray!20] (0,3) -- (3,0) -- (2,-5);
\fill[gray!20] (0,3) -- (-1,-2) -- (2,-5) -- (0,0);
\fill[gray!50] (0,3) -- (2,-5) -- (0,0);
\draw[line width=0.2mm] (0,3) -- (3,0) -- (2,-5) -- (-1,-2) -- (0,3);
\draw[line width=0.2mm] (0,3) -- (2,-5);
\draw[line width=0.2mm] (0,0) -- (0,3);
\draw[line width=0.2mm] (0.8,0) -- (3,0);
\draw[dashed, line width=0.2mm] (0,0) -- (0.8, 0);
\draw[line width=0.2mm] (0,0) -- (2,-5);
\draw[line width=0.2mm] (0,0) -- (-1,-2);
\draw[line width=0.2mm] (-0.7, -0.35) -- (-2, -1);
\draw[dashed, line width=0.2mm] (0,0) -- (-0.7, -0.35);
\draw (0,0) node{$\bullet$};
\draw (-1,-2) node{$\bullet$} node[left]{$u_1$};
\draw (2,-5) node{$\bullet$} node[right]{$u_2$};
\draw (3,0) node{$\bullet$} node[above]{$u_3$};
\draw (0,3) node{$\bullet$} node[right]{$u_4$};
\draw (-2,-1) node{$\bullet$} node[above]{$u_0$};
\end{tikzpicture}
\caption{Fan of $\Sigma$}
\label{fig:flip1a}
\end{subfigure}
\hfill
\begin{subfigure}{0.3\textwidth}
\begin{tikzpicture}[scale=0.6]
\fill[gray!20] (0,3) -- (3,0) -- (2,-5) -- (-1, -2);
\draw[line width=0.2mm] (0,3) -- (3,0) -- (2,-5) -- (-1,-2) -- (0,3);
\draw[line width=0.2mm] (0,0) -- (0,3);
\draw[line width=0.2mm] (0,0) -- (3,0);
\draw[line width=0.2mm] (0,0) -- (2,-5);
\draw[line width=0.2mm] (0,0) -- (-1,-2);
\draw[line width=0.2mm] (-0.7, -0.35) -- (-2, -1);
\draw[dashed, line width=0.2mm] (0,0) -- (-0.7, -0.35);
\draw (0,0) node{$\bullet$};
\draw (-1,-2) node{$\bullet$} node[left]{$u_1$};
\draw (2,-5) node{$\bullet$} node[right]{$u_2$};
\draw (3,0) node{$\bullet$} node[above]{$u_3$};
\draw (0,3) node{$\bullet$} node[right]{$u_4$};
\draw (-2,-1) node{$\bullet$} node[above]{$u_0$};
\end{tikzpicture}
\caption{Fan of $\Sigma_0$}
\label{fig:flip1b}
\end{subfigure}
\hfill
\begin{subfigure}{0.3\textwidth}
\begin{tikzpicture}[scale=0.6]
\fill[gray!20] (0,3) -- (3,0) -- (0,0) -- (-1, -2);
\fill[gray!20] (-1,-2) -- (3,0) -- (2,-5);
\fill[gray!50] (-1,-2) -- (0,0) -- (3,0);
\draw[line width=0.2mm] (0,3) -- (3,0) -- (2,-5) -- (-1,-2) -- (0,3);
\draw[line width=0.2mm] (-1, -2) -- (3,0);
\draw[line width=0.2mm] (0,0) -- (0,3);
\draw[line width=0.2mm] (0,0) -- (3,0);
\draw[line width=0.2mm] (0.5,-1.25) -- (2,-5);
\draw[dashed, line width=0.2mm] (0,0) -- (0.5, -1.25);
\draw[line width=0.2mm] (0,0) -- (-1,-2);
\draw[line width=0.2mm] (-0.7, -0.35) -- (-2, -1);
\draw[dashed, line width=0.2mm] (0,0) -- (-0.7, -0.35);
\draw (0,0) node{$\bullet$};
\draw (-1,-2) node{$\bullet$} node[left]{$u_1$};
\draw (2,-5) node{$\bullet$} node[right]{$u_2$};
\draw (3,0) node{$\bullet$} node[above]{$u_3$};
\draw (0,3) node{$\bullet$} node[right]{$u_4$};
\draw (-2,-1) node{$\bullet$} node[above]{$u_0$};
\end{tikzpicture}
\caption{Fan of $\Sigma'$}
\label{fig:flip1c}
\end{subfigure}

\caption{Fans of varieties given in Section \ref{sec:example}}
\end{figure}
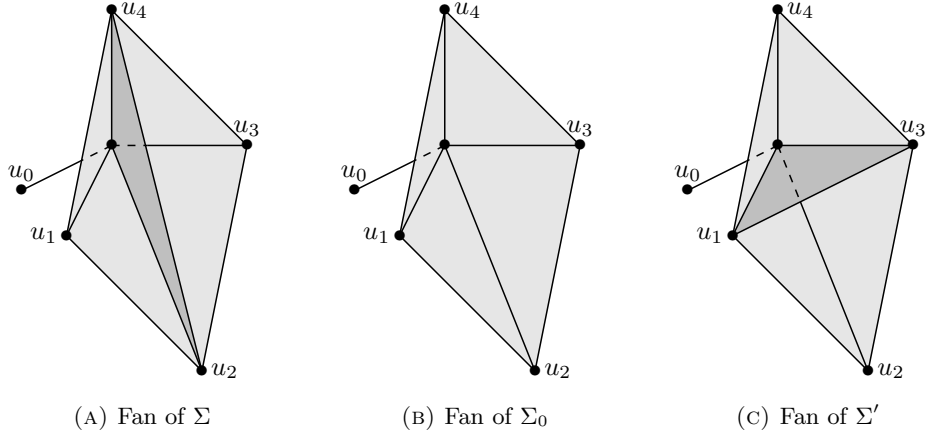

\begin{lemma}[Intersections of divisors]
\label{lem:intersection-example}
$\,$
\begin{enumerate}
\item On $X$ and $X'$, we have:
$$
\begin{array}{l}
D_1 \cdot D_0^2 = \dfrac{1}{2} \qquad D_2 \cdot D_0^2 = \dfrac{1}{4} \qquad
D_4 \cdot D_0^2 = 1 \qquad D_0 \cdot D_0^2 = \dfrac{3}{4}
\\
D_1 \cdot (D_0 \cdot D_2) = \dfrac{1}{2} \qquad D_4 \cdot (D_0 \cdot D_2)=0 \qquad
D_2 \cdot (D_0 \cdot D_2)= - \dfrac{1}{4}.
\end{array}
$$
\item On $X$, we have
$$
D_1 \cdot D_2^2 = \dfrac{1}{2} \qquad D_4 \cdot D_2^2 = -1 \qquad
D_0 \cdot D_2^2 = \dfrac{-1}{4} \qquad D_2 \cdot D_2^2= \dfrac{-3}{4}.
$$
\item On $X'$, we have
$$
D_1 \cdot D_2^2 = \dfrac{-1}{2} \qquad D_4 \cdot D_2^2 = 0 \qquad
D_0 \cdot D_2^2 = \dfrac{-1}{4} \qquad D_2 \cdot D_2^2= \dfrac{1}{4}.
$$
\end{enumerate}
\end{lemma}

\begin{proof}
The lemma follows from Formula (\ref{eq:intersection-toric-divisors}) and Table \ref{tab:index-sublatice}.
\begin{table}[t]
\begin{tabular}{|c|c|c|c|}
\hline 
$\Z u_0 + \Z u_1 + \Z u_2$ & $\Z u_0 + \Z u_2 + \Z u_3$ & $\Z u_0 + \Z u_3 + \Z u_4$ &
$\Z u_0 + \Z u_4 + \Z u_1$\\ 
\hline 
$2$ & $2$ & $1$ & $1$ \\ 
\hline 
\end{tabular} 

\vspace{0.5cm}
\begin{tabular}{|c|c||c|c|}
\hline 
$\Z u_2 + \Z u_4 + \Z u_1$ & $\Z u_2 + \Z u_4 + \Z u_3$ & $\Z u_1 + \Z u_3 + \Z u_2$ &
$\Z u_1 + \Z u_3 + \Z u_4$\\ 
\hline 
$1$ & $1$ & $1$ & $1$ \\ 
\hline 
\end{tabular}

\caption{Sublattices and their index in $\Z^3$}
\label{tab:index-sublatice}
\end{table}
We show the first point to illustrate the computations, the other intersection numbers follow in the same way.
\begin{align*}
D_1 \cdot D_0^2 & = D_1 \cdot D_0 \cdot (D_2 + D_3) = D_1 \cdot D_0 \cdot D_2 +
D_1 \cdot D_0 \cdot D_3 = \dfrac{1}{2}
\\
D_2 \cdot D_0^2 &= \dfrac{1}{2} D_2 \cdot D_0 \cdot (D_3 + D_4) = \dfrac{1}{4}
\\
D_4 \cdot D_0^2 & = D_4 \cdot D_0 \cdot (D_2 + D_3) = 1
\\
D_0 \cdot D_0^2 &= \dfrac{1}{2}(D_3 + D_4) \cdot D_0^2 = \dfrac{1}{4} + \dfrac{1}{2}=
\dfrac{3}{4}.
\end{align*}
\end{proof}

We assume that $\cE = \cT_X$. By \cite[Corollary 2.2.17]{DDK20}, the family of filtrations of the tangent sheaf $\cT_X$ of $X$ is given by:
$$
E^{\rho}(j) = \left\lbrace
\begin{array}{ll}
0 & \text{if}~ j< -1 \\
\Span(u_{\rho}) & \text{if}~ j = -1 \\
N \otimes_{\Z} \C & \text{if} ~ j> -1
\end{array}
\right..
$$
In this case, the inequalities of Proposition \ref{prop:stabilityequivariantcase} become
$$
\dfrac{1}{\dim(F)} \sum_{u_\rho \in F} \deg_{L}(D_\rho) \stackrel{(\leq)}{<}
\dfrac{1}{n} \sum_{\rho \in \Sigma(1)} \deg_{L}(D_\rho).
$$
Therefore, to study the stability of $\cE$, it suffices to consider the vector spaces $F$ given by
$$
F = \Span(u_\rho: \rho \in \Gamma)
$$
with $\Gamma \subseteq \Sigma(1)$.

Let $L_0 = D_0$ be an ample Cartier divisor on $X_0$. We have
$$
\mu_{L}(\phi_* \cE) = 1 = \dfrac{1}{\dim(F_1)} \sum_{u_\rho \in F_1} \deg_{L}(D_\rho) =
\dfrac{1}{\dim(F_2)} \sum_{u_\rho \in F_2} \deg_{L}(D_\rho)
$$
where $F_1 = \Span(u_4)$ and $F_2 = \Span(u_0, u_2, u_4)$. As for any vector subspace
$F \subsetneq E$ such that $F \notin \{F_1, F_2 \}$, one has
$$
\dfrac{1}{\dim(F)} \sum_{u_\rho \in F} \deg_{L}(D_\rho) < 1,
$$
we deduce that $\phi_* \cE$ is semistable. We denote respectively by $\cF$ and $\cF'$ the
subsheaf of $\phi_* \cE$ given by the families of filtrations $(F_1, F_1 \cap E^\rho(i))$
and $(F_2, F_2 \cap E^\rho(i))$. We have
$$
0 \subseteq \cF \subseteq \cF' \subseteq \cE.
$$
Let $$D_{+}= D_2 + D_4.$$
We have 
$$D_{+} \sim_{\rm lin} D_0 + 2 D_2.$$
So
$$
L_{-\ep} = \phi^*L_0 - \ep D_+ \sim_{\rm lin} (1-\ep)D_0 - 2 \ep D_2 =
(1- \ep) \left(D_0 - \dfrac{2 \ep}{1 - \ep} D_2 \right)
$$
and
$$
L_{\ep} = (\phi')^*L_0 + \ep D_+ \sim_{\rm lin} (1+\ep) D_0 + 2 \ep D_2 =
(1+ \ep) \left(D_0 + \dfrac{2 \ep}{1 + \ep} D_2 \right).
$$
According to \ref{point:theo-main3} and \ref{point:theo-main4} of Theorem \ref{theo:main},
to check stability of $\cE$ with respect to $L_{-\ep}$, it suffices to compare
$\mu_{L_{-\ep}}(\cE)$ with $\mu_{L_{-\ep}}(\phi^* \cF)$ and $\mu_{L_{-\ep}}(\phi^* \cF')$.
%
%
By Lemma \ref{lem:intersection-example}, on $X$ we have
\begin{equation}\label{eq:degree-divisorX}
\begin{aligned}
& \deg_{L_{-\ep}}(D_1)= \dfrac{1}{2} - 3 \ep + \dfrac{9}{2} \ep^2 & \qquad &
\deg_{L_{-\ep}}(D_2) = \dfrac{1}{4} + \dfrac{1}{2} \ep - \dfrac{15}{4} \ep^2
\\
& \deg_{L_{-\ep}}(D_4) = 1 - 2 \ep - 3 \ep^2 & &
\deg_{L_{-\ep}}(D_0) = \dfrac{3}{4} - \dfrac{5}{2} \ep + \dfrac{3}{4} \ep^2;
\end{aligned}
\end{equation}
and on $X'$ we have
\begin{equation}\label{eq:degree-divisorX2}
\begin{aligned}
& \deg_{L_{\ep}}(D_1)= \dfrac{1}{2} + 3 \ep + \dfrac{1}{2} \ep^2 & \qquad &
\deg_{L_{\ep}}(D_2) = \dfrac{1}{4} - \dfrac{1}{2} \ep + \dfrac{1}{4} \ep^2
\\
& \deg_{L_{\ep}}(D_4) = 1 + 2 \ep + \ep^2 & &
\deg_{L_{\ep}}(D_0) = \dfrac{3}{4} + \dfrac{5}{2} \ep + \dfrac{3}{4} \ep^2.
\end{aligned}
\end{equation}
By Formulas (\ref{eq:degree-divisorX}), we get
$$
\mu_{L_{-\ep}}(\cE) = 1 - \dfrac{10}{3} \ep + \ep^2
\quad \text{and} \quad
\mu_{L_{-\ep}}(\cF) = \mu_{L_{-\ep}}(\cF') = 1 - 2 \ep - 3 \ep^2.
$$
On the other hand, by Formulas (\ref{eq:degree-divisorX2}), we have
$$
\mu_{L_{\ep}}(\psi_* \cE) = 1 + \dfrac{10}{3} \ep + \ep^2 \quad \text{and} \quad
\mu_{L_{\ep}}((\phi')^* \cF) = \mu_{L_{\ep}}((\phi')^* \cF') = 1 + 2 \ep + \ep^2.
$$
Hence, there is $\ep_0$ such that for any $\ep \in (0, \ep_0) \cap \Q$:
\begin{itemize}
\item the tangent sheaf $\cT_X$ is unstable with respect to $L_{-\ep}$;
\item the tangent sheaf $\cT_{X'}$ is stable with respect to $L_{\ep}$.
\end{itemize}
This example clearly shows point \ref{point:theo-main4} of Theorem \ref{theo:main}.

\section{Flips and stability for logarithmic subcategories}
\label{sec:secondtheorem}
Consider a toric flip
$$
 \begin{tikzcd}
  X   &  & X' \arrow[ddl, rightarrow,"{\phi'}"']\arrow[ll,dashleftarrow, "{\psi}"'] \\
   & & \\
  & X_0 \arrow[uul, leftarrow, "{\phi}"'] &  
\end{tikzcd}
 $$
 as in Section \ref{sec:toricflips}. Fix $\Delta\subset \Sigma(1)$, and introduce the divisor
 $$
 D:=\sum_{\rho\in \Delta} D_\rho
 $$
seen as a divisor on $X$, $X'$ and $X_0$. We consider the equivalent categories $\Ref^T(X,D)$ and $\Ref^T(X',D)$ as in Section \ref{sec:logcase}. We will say that 
$$\psi_* : \Ref^T(X,D) \to \Ref^T(X',D)$$
{\it preserves polystability} for the pair of polarisations $(\alpha, \alpha')\in \Pic(X)_\Q\times \Pic(X')_\Q$ if for any $\cE\in\mathrm{Obj}(\Ref^T(X,D))$, $\cE$ is $\alpha$-polystable if and only if $\psi_*\cE$ is $\alpha'$-polystable.
\begin{theorem}
 \label{theo:second}
 The following assertions are equivalent, for a pair of ample classes $(\alpha,\alpha')\in \Pic(X)_\Q\times \Pic(X')_\Q$.
 \begin{enumerate}[label=(\roman*)]
  \item The functor $\psi_* : \Ref^T(X,D) \to \Ref^T(X',D)$ preserves polystability for $(\alpha,\alpha')$. \label{point:theo-second1}
  \item There is $c \in\Q_{>0}$ such that for all $\rho \notin \Delta$, $\deg_\alpha D_\rho = c\; \deg_{\alpha'} D_\rho'$. \label{point:theo-second2}
 \end{enumerate}
\end{theorem}
We stated Theorem \ref{theo:second} for flips, but the proof that follows works for any small birational equivariant map between two normal toric varieties (that is a birational and equivariant map that is an isomorphism away from codimension $2$ subsets). We thank the anonymous referee for pointing this out.
\begin{proof}
Recall the formula
$$
 \mu_\alpha(\cE)=-\frac{1}{\rank(\cE)}\sum_{\rho\in\Sigma(1)} \iota_\rho(\cE)\, \deg_\alpha(D_\rho)
$$
for the slope, with 
$$
 \iota_\rho(\cE):=\sum_{i\in\Z} i \left(\dim(E^\rho(i))-\dim (E^\rho(i-1))\right)
 $$
 where $(E,E^\rho(\bullet))_{\rho\in\Sigma(1)}$ stands for the family of filtrations associated to $\cE$. Then, by definition of the logarithmic category $\Ref^T(X,D)$ (see Equation (\ref{eq:definingconditionlogcategory})), for any $\cE\in\mathrm{Obj}(\Ref^T(X,D))$, and any $\rho\in\Delta$, we have
 $$
 \iota_\rho(\cE)= a_\rho \dim(E),
 $$
 so the slope reads
 $$
 \mu_\alpha(\cE)=-\frac{1}{\rank(\cE)}\sum_{\rho\notin\Delta} \iota_\rho(\cE)\, \deg_\alpha(D_\rho) - \sum_{\rho \in \Delta} a_\rho \deg_{\alpha}(D_\rho).
 $$
 Note also that by construction, for any $\rho$,
 $$
 \iota_\rho(\cE)=\iota_\rho(\psi_*\cE).
 $$
 Then, $(ii) \Rightarrow (i)$ is straightforward.
%
%
To prove $(i) \Rightarrow (ii)$, we argue as in \cite[proof of Proposition 4.8]{CT22}, and consider for any pair $(\rho_1,\rho_2)\in (\Sigma(1) \setminus \Delta)^2$ the polystable toric sheaf 
 $$
 \cE=\cO_X(d\deg_\alpha(D_{\rho_2}) D_{\rho_1})\oplus\cO_X(d\deg_\alpha(D_{\rho_1}) D_{\rho_2}),
 $$
 where $d$ is the common denominator of $\deg_\alpha(D_{\rho_2})$ and $\deg_\alpha(D_{\rho_1})$.
 Its image by $\psi_*$ is 
 $$
 \psi_*\cE=\cO_{X'}(d\deg_\alpha(D_{\rho_2}) D_{\rho_1}')\oplus\cO_{X'}(d\deg_\alpha(D_{\rho_1}) D_{\rho_2}').
 $$
 As $\cE\in\mathrm{Obj}(\Ref^T(X,D))$ (see e.g. \cite[Example 2.2.13]{DDK20} for the family of filtrations of rank one toric sheaves), if $\psi$ preserves polystability, we must have
 $$
 \deg_\alpha(D_{\rho_2}) \deg_{\alpha'} (D_{\rho_1}')=\deg_\alpha(D_{\rho_1}) \deg_{\alpha'} (D_{\rho_2}').
 $$
 The result follows.
\end{proof}

\begin{remark}
The reason for considering logarithmic subcategories is the following.
If one considers the case $\Delta = \varnothing$, i.e the full $\Ref^T(X)$, then for no choice of $(\alpha,\alpha')$ polystability is preserved.
Indeed, by \ref{point:theo-second2} of Theorem \ref{theo:second}, for any $\rho \in \Sigma(1)$, one must have
$$
\deg_\alpha D_\rho = c\; \deg_{\alpha'} D_\rho'.
$$
Up to scale, we can assume $c=1$. But then, a result due to Minkowski (\cite[p.455]{Sch}), translated into the toric setting in \cite[Proposition 5.3 and Corollary 5.4]{CT22}, implies that the polytope associated to $(X,\alpha)$ equals the polytope associated to $(X',\alpha')$, which is absurd, as $X$ and $X'$ are not isomorphic.
\end{remark}

\begin{remark}
Let $L_0$ be an ample Cartier divisor on $X_0$ and let $\ep >0$ be such that the divisors $L_{-\ep}= \phi^* L_0 - \ep D_{+}$ on $X$ and $L_{\ep}= (\phi')^* L_0 + \ep D_{+}$ on $X'$ define $\Q$-ample Cartier divisors.
If $\psi_* : \Ref^T(X,D) \to \Ref^T(X',D)$ preserves polystability for $(L_{-\ep}, L_{\ep})$ then according to \ref{point:theo-second2} of Theorem \ref{theo:second}, for any $\rho \notin \Delta$, we have
$$
\deg_{L_{-\ep}}(D_\rho) = \deg_{L_{\ep}}(D_\rho).
$$
because the constant term in the $\ep$-expansions of $\deg_{L_{-\ep}}(D_\rho)$ and $\deg_{L_{-\ep}}(D_\rho)$ are equal. Therefore,
$$
\Delta \subseteq \left\lbrace \rho \in \Sigma(1): D_\rho \cdot D_{+} \cdot (\phi^* L_0)^{n-2} \neq 0 \right\rbrace.
$$
\end{remark}

\bibliographystyle{plain}	
 \bibliography{Flips}

\end{document}